\documentclass[a4paper]{amsart}
\usepackage{amsmath, amssymb}
\usepackage{hyperref}
\usepackage[margin=2cm]{geometry}
\numberwithin{equation}{section}
\newtheorem{theorem}{Theorem}[section]

\newtheorem{lemma}{Lemma}[section]

\theoremstyle{remark}
\newtheorem{remark}{Remark}[section]

\title[Notes on the norm of pre-Schwarzian derivatives]
{Notes on the norm of pre-Schwarzian derivatives
of certain analytic functions}



\begin{document}
\begin{abstract}
In this paper, we obtain sharp bounds for the norm of pre--Schwarzian derivatives
of certain analytic functions.
Initially this problem was handled by H. Rahmatan, Sh. Najafzadeh and A. Ebadian [Stud Univ Babe\c{s}--Bolyai Math {\bf61}(2): 155--162, 2016]. We pointed out that the proofs by Rahmatan et al. are incorrect and present correct proofs.
\end{abstract}

\author[ R. Kargar] {R. Kargar}
\address{Young Researchers and Elite Club,
Ardabil Branch, Islamic Azad University, Ardabil, Iran}
\email{rkargar1983@gmail.com, rkargar@pnu.ac.ir}

\maketitle

\section{Introduction}
Let $\Delta$ be the open unit disc on the complex plane $\mathbb{C}$. Let $\mathcal{H}$ be the family of all analytic functions and $\mathcal{A}\subset \mathcal{H}$ be the family of all normalized functions in $\Delta$. We denote by $\mathcal{U}$ the class of all univalent functions in $\Delta$ and denote by $\mathcal{LU}\subset \mathcal{H}$ the class of all locally univalent functions in $\Delta$. For a $f\in\mathcal{LU}$, we consider the following norm
\begin{equation*}
  ||f||=\sup_{z\in\Delta}(1-|z|^2)\left|\frac{f''(z)}{f'(z)}\right|,
\end{equation*}
where the quantity $f''/f$ is often referred to as pre--Schwarzian derivative of $f$ and in the theory of Teichm\"{u}ller spaces is considered as element of complex Banach spaces. We remark that $||f||<\infty$ if, and only if, $f$ is uniformly locally univalent in $\Delta$. Note that, $||f||\leq 6$ if $f$ is univalent in $\Delta$ and, conversely, $f$ is univalent in $\Delta$ if
$||f|\leq 1$. Both of these bounds are sharp, see \cite{BecPom}. For more geometric properties of the function $f$ relating the norm, see \cite{choi, kimsug, PonSS} and the references therein.

We say that a function $f$ is subordinate to $g$, written by $f(z)\prec g(z)$ or $f\prec g$ where $f$ and $g$ belonging to the class $\mathcal{A}$, if there exists a Schwarz function $w(z)$ is analytic in $\Delta$ with
  \begin{equation*}
    w(0)=0\quad{\rm and}\quad |w(z)|<1\quad(z\in\Delta),
  \end{equation*}
such that $f(z)=g(w(z))$ for all $z\in\Delta$.

In the sequel, we recall two definition which are certain subclasses of analytic and normalized functions $\mathcal{A}$. First, we say that a function $f\in\mathcal{A}$ belongs to the class $\mathcal{S}(\alpha,\beta)$ if it satisfies the following two--sided inequality
\begin{equation*}
  \alpha<{\rm Re}\left\{\frac{zf'(z)}{f(z)}\right\}<\beta\quad(z\in\Delta),
\end{equation*}
where $0\leq \alpha<1$ and $\beta>1$. The class $\mathcal{S}(\alpha,\beta)$ was introduced by Kuroki and Owa (cf. \cite{KO2011}). Also, we say that a function $f\in\mathcal{A}$ belongs to the class $\mathcal{V}(\alpha,\beta)$ if
\begin{equation*}
  \alpha<{\rm Re}\left\{\left(\frac{z}{f(z)}\right)^2f'(z)\right\}<\beta\quad(z\in\Delta).
\end{equation*}
The class $\mathcal{V}(\alpha,\beta)$ was first introduced by Kargar et al., see \cite{KES(Siberian)}.

Since the convex univalent function
\begin{equation}\label{P}
  P_{\alpha,\beta}(z)=1+\frac{(\beta-\alpha)i}{\pi}\log\left(\frac{1-e^{i\phi}z}{1-z}\right)
  \quad(z\in\Delta),
\end{equation}
where
  \begin{equation}\label{phi}
  \phi:=\frac{2\pi(1-\alpha)}{\beta-\alpha},
  \end{equation}
maps $\Delta$ onto the domain $\Omega=\{\omega: \alpha<{\rm Re}\{\omega\}<\beta\}$ conformally, thus we have.
\begin{lemma}\label{lem S alpha beta}
{\rm(}\cite[Lemma 1.3]{KO2011}{\rm )}
  Let $0\leq \alpha<1$ and $\beta>1$. Then $f\in\mathcal{S}(\alpha,\beta)$ if, and only if,
  \begin{equation*}
    \frac{zf'(z)}{f(z)}\prec 1+\frac{(\beta-\alpha)i}{\pi}\log\left(\frac{1-e^{i\phi}z}{1-z}\right)
  \quad(z\in\Delta),
\end{equation*}
where $\phi$ is defined in \eqref{phi}.
\end{lemma}
\begin{lemma}\label{lem V alpha beta}
{\rm(}\cite[Lemma 1.1]{KES(Siberian)}{\rm )}
  Let $0\leq \alpha<1$ and $\beta>1$. Then $f\in\mathcal{V}(\alpha,\beta)$ if, and only if,
  \begin{equation*}
    \left(\frac{z}{f(z)}\right)^2f'(z)\prec  1+\frac{(\beta-\alpha)i}{\pi}\log\left(\frac{1-e^{i\phi}z}{1-z}\right)
  \quad(z\in\Delta),
\end{equation*}
where $\phi$ is defined in \eqref{phi}.
\end{lemma}
Rahmatan et al. (see \cite{RNE}) estimated the norm of pre--Schwarzian derivatives of the function $f$ where $f$ belong to the classes $\mathcal{S}(\alpha,\beta)$ and $\mathcal{V}(\alpha,\beta)$. Both estimates and proofs are incorrect. Indeed, the estimates of $||f||$ were wrongly proven by Rahmatan, Najafzadeh and Ebadian are in the following form:\\ \\
{\bf Theorem A:} For $0\leq \alpha<1<\beta$, if $f\in\mathcal{S}(\alpha,\beta)$, then
\begin{equation*}
  ||f||\leq \frac{2(\beta-\alpha)}{\pi}\left(1-e^{2\pi i\frac{1-\alpha}{\beta-\alpha}}\right).
\end{equation*}
{\bf Theorem B:}
For $0\leq \alpha<1<\beta$, if $f\in\mathcal{V}(\alpha,\beta)$, then
\begin{equation*}
  ||f||\leq \frac{3(\beta-\alpha)}{\pi}\left(1-e^{2\pi i\frac{1-\alpha}{\beta-\alpha}}\right).
\end{equation*}
First, note that both bounds are complex numbers.

In this paper we give the best estimate for $||f||$ when $f\in\mathcal{S}(\alpha,\beta)$ and disprove the Theorem B. However, we show that $||f||<\infty$ when $f\in\mathcal{V}(\alpha,\beta)$.

\section{The Main Results}
The first result of the paper is the following.
\begin{theorem}
  Let $0\leq \alpha<1$ and $\beta>1$. If a function $f$ belongs to the class $\mathcal{S}(\alpha,\beta)$, then
  \begin{equation}\label{norm S alpha beta}
    ||f||\leq \frac{2(\beta-\alpha)}{\pi}\sqrt{4 \sin^2(\phi/2)+2\pi^2}-\frac{4\sin(\phi/2)}{\sqrt{4 \sin^2(\phi/2)+2\pi^2}},
  \end{equation}
  where $\phi$ is defined in \eqref{phi}.
  The result is sharp.
\end{theorem}
\begin{proof}
  Let that $0\leq \alpha<1$, $\beta>1$ and $\phi$ be given by \eqref{phi}. If
  $f\in\mathcal{S}(\alpha,\beta)$, by Lemma \ref{lem S alpha beta}, then we have
  \begin{equation}\label{z f prime f sub P alpha beta}
    \frac{zf'(z)}{f(z)}\prec 1+\frac{(\beta-\alpha)i}{\pi}\log\left(\frac{1-e^{i\phi}z}{1-z}\right)\quad(z\in\Delta).
  \end{equation}
  The above subordination relation \eqref{z f prime f sub P alpha beta} implies that
    \begin{equation*}
    \frac{zf'(z)}{f(z)}= 1+\frac{(\beta-\alpha)i}{\pi}\log\left(\frac{1-e^{i\phi}w(z)}{1-w(z)}\right)\quad(z\in\Delta),
  \end{equation*}
  or equivalently
  \begin{equation}\label{log z f prime f equal P alpha beta w}
   \log\left\{\frac{zf'(z)}{f(z)}\right\}=\log\left\{ 1+\frac{(\beta-\alpha)i}{\pi}\log\left(\frac{1-e^{i\phi}w(z)}{1-w(z)}\right)\right\}\quad(z\in\Delta),
  \end{equation}
  where $w(z)$ is the Schwarz function.
  From \eqref{log z f prime f equal P alpha beta w}, differentiating on both sides, after simplification, we obtain
  \begin{equation}\label{f second f prime}
    \frac{f''(z)}{f'(z)}=\frac{(\beta-\alpha)i}{\pi}\left[\frac{1}{z}\log\left(\frac{1-e^{i\phi}w(z)}
    {1-w(z)}\right)+\frac{(1-e^{i\phi})w'(z)}{(1-w(z))(1-e^{i\phi}w(z))\left(1+\frac{(\beta-\alpha)i}
    {\pi}\log\left(\frac{1-e^{i\phi}w(z)}
    {1-w(z)}\right)\right)}\right].
  \end{equation}
  It is well--known that $|w(z)|\leq|z|$ (cf. \cite{Duren}) and also by the Schwarz--Pick lemma, for a Schwarz function the following inequality
  \begin{equation}\label{w prime}
    |w'(z)|\leq \frac{1-|w(z)|^2}{1-|z|^2}\quad(z\in\Delta),
  \end{equation}
  holds. Also, we know that if $\log$ is the principal branch of the complex logarithm, then we have
  \begin{equation}\label{log z}
    \log z= \ln |z|+i \arg z\quad(z\in\Delta\setminus\{0\}, -\pi<\arg z\leq \pi).
  \end{equation}
  Therefore, by the above equation \eqref{log z}, it is well--known that if $|z|\geq1$, then
  \begin{equation}\label{ineq log 1}
    |\log z|\leq \sqrt{|z-1|^2+\pi^2},
  \end{equation}
  while for $0<|z|<1$, we have
    \begin{equation}\label{ineq log 2}
    |\log z|\leq \sqrt{\left|\frac{z-1}{z}\right|^2+\pi^2}.
  \end{equation}
  Thus, it is natural to distinguish the following cases.\\
  {\bf Case 1:} $\left|\frac{1-e^{i\phi}w(z)}
    {1-w(z)}\right|\geq 1$.\\
    By \eqref{ineq log 1}, we have
    \begin{align}\label{estimate abs log geq 1}
      \left|\log\left(\frac{1-e^{i\phi}w(z)}
    {1-w(z)}\right)\right|&\leq \sqrt{\left|\frac{1-e^{i\phi}w(z)}
    {1-w(z)}-1\right|^2+\pi^2}\nonumber\\
    &=\frac{\sqrt{|1-e^{i\phi}|^2|w(z)|^2+\pi^2|1-w(z)|^2}}{|1-w(z)|}\nonumber\\
    &\leq\frac{\sqrt{4 \sin^2(\phi/2)|w(z)|^2+\pi^2(1+|w(z)|^2)}}{1-|w(z)|}\nonumber\\
    &\leq\frac{\sqrt{4 \sin^2(\phi/2)|z|^2+\pi^2(1+|z|^2)}}{1-|z|}
    \end{align}
    for all $z\in\Delta$.
  We note that the above inequality is well defined also for $z=0$.
  Thus from \eqref{f second f prime}, \eqref{w prime} and \eqref{estimate abs log geq 1}, we get
  \begin{align*}
    &\quad\left|\frac{f''(z)}{f'(z)}\right|\\ &=\left|\frac{(\beta-\alpha)i}{\pi}\left[\frac{1}{z}\log\left(\frac{1-e^{i\phi}w(z)}
    {1-w(z)}\right)+\frac{(1-e^{i\phi})w'(z)}{(1-w(z))(1-e^{i\phi}w(z))\left(1+\frac{(\beta-\alpha)i}
    {\pi}\log\left(\frac{1-e^{i\phi}w(z)}
    {1-w(z)}\right)\right)}\right]\right| \\
    &\leq \frac{(\beta-\alpha)}{\pi}\left[\frac{1}{|z|}\left|\log\left(\frac{1-e^{i\phi}w(z)}
    {1-w(z)}\right)\right|+\frac{\left|1-e^{i\phi}\right||w'(z)|}{\left|1-w(z)\right|
    \left|1-e^{i\phi}w(z)\right|\left(1-\frac{(\beta-\alpha)}
    {\pi}\left|\log\left(\frac{1-e^{i\phi}w(z)}{1-w(z)}\right)\right|\right)}\right]\\
    &\leq\frac{(\beta-\alpha)}{\pi}\left[\frac{1}{|z|}\left\{\frac{\sqrt{4 \sin^2(\phi/2)|z|^2+\pi^2(1+|z|^2)}}{1-|z|}
    \right\}+\frac{2\sin(\phi/2)}{1-|z|-\frac{(\beta-\alpha)}{\pi}\sqrt{4 \sin^2(\phi/2)|z|^2+\pi^2(1+|z|^2)}
    }.\frac{1+|z|}
    {1-|z|^2}\right].
  \end{align*}
  However, we obtain
  \begin{align*}
    ||f||&=\sup_{z\in\Delta}(1-|z|^2)\left|\frac{f''(z)}{f'(z)}\right|\\
    &\leq\sup_{z\in\Delta}\left\{\frac{(\beta-\alpha)}{\pi}\left[
    \frac{1+|z|}{|z|}\sqrt{4 \sin^2(\phi/2)|z|^2+\pi^2(1+|z|^2)}+\frac{2\sin(\phi/2)(1+|z|)}{1-|z|-\frac{(\beta-\alpha)}{\pi}\sqrt{4 \sin^2(\phi/2)|z|^2+\pi^2(1+|z|^2)}
    }\right]\right\}\\
    &=\frac{2(\beta-\alpha)}{\pi}\sqrt{4 \sin^2(\phi/2)+2\pi^2}-\frac{4\sin(\phi/2)}{\sqrt{4 \sin^2(\phi/2)+2\pi^2}}
  \end{align*}
  and concluding the inequality \eqref{norm S alpha beta}.\\
    {\bf Case 2:} $\left|\frac{1-e^{i\phi}w(z)}
    {1-w(z)}\right|< 1$.\\
  By \eqref{ineq log 2}, we have
      \begin{align*}
      \left|\log\left(\frac{1-e^{i\phi}w(z)}
    {1-w(z)}\right)\right|&\leq \sqrt{\left|\frac{\frac{1-e^{i\phi}w(z)}
    {1-w(z)}-1}{\frac{1-e^{i\phi}w(z)}
    {1-w(z)}}\right|^2+\pi^2}\\
    &=\frac{\sqrt{|1-e^{i\phi}|^2|w(z)|^2+\pi^2|1-e^{i\phi}w(z)|^2}}{|1-e^{i\phi}w(z)|}\\
    &\leq\frac{\sqrt{4 \sin^2(\phi/2)|w(z)|^2+\pi^2(1+|w(z)|^2)}}{1-|w(z)|}\quad(|e^{i\phi}|=1)\\
    &\leq\frac{\sqrt{4 \sin^2(\phi/2)|z|^2+\pi^2(1+|z|^2)}}{1-|z|}.
    \end{align*}
    Since in the Cases 1 and 2 we have the equal estimates for 
    \begin{equation*}
    \left|\log\left(\frac{1-e^{i\phi}w(z)}
    {1-w(z)}\right)\right|,
    \end{equation*}
  therefore, in this case also, the desired result will be achieved.
  For the sharpness, consider the function $f_{\alpha,\beta}(z)$ as follows
  \begin{align*}
    f_{\alpha,\beta}(z)&=z\exp\left\{\frac{(\beta-\alpha)i}{\pi}\int_{0}^{z}\frac{1}{\xi}
    \log\left(\frac{1-e^{i\phi}\xi}{1-\xi}\right)d\xi\right\}\\
    &=z+\frac{(\beta-\alpha)i}{\pi}\left(1-e^{i\phi}\right)z^2+\cdots,
  \end{align*}
  where $\phi$ is defined in \eqref{phi}, $0\leq \alpha<1$ and $\beta>1$. A simple calculation, gives us
  \begin{equation*}
        \frac{zf'_{\alpha,\beta}(z)}{f_{\alpha,\beta}(z)}= 1+\frac{(\beta-\alpha)i}{\pi}\log\left(\frac{1-e^{i\phi}z}{1-z}\right)\quad(z\in\Delta)
  \end{equation*}
  and thus $f_{\alpha,\beta}(z)\in\mathcal{S}(\alpha,\beta)$. With the same proof as above we get the desired result. Also, the result is sharp for a rotation of the function $ f_{\alpha,\beta}(z)$ as follows:
  \begin{equation*}
   \mathfrak{f}_{\alpha,\beta}(z)=z\exp\left\{\frac{(\beta-\alpha)i}{\pi}\int_{0}^{z}\frac{1}{\xi}
    \log\left(\frac{1-e^{i\phi}\xi}{1-e^{-i\phi}\xi}\right)d\xi\right\}.
  \end{equation*}
  This is the end of proof.
\end{proof}
\begin{remark}
  In the Theorem B, the authors estimated $||f||$ when $f\in\mathcal{V}(\alpha,\beta)$. But in the proof of this theorem \cite[p. 160]{RNE}, wrongly, they used from the following equation
  \begin{equation*}
    \frac{zf'(z)}{f(z)}=P_{\alpha,\beta}(w(z)),
  \end{equation*}
  where $P_{\alpha,\beta}$ is defined in \eqref{P}. This means that $f$, simultaneously, belonging to the class $\mathcal{S}(\alpha,\beta)$ and $\mathcal{V}(\alpha,\beta)$. Next, we show that the best estimate for $||f||$ when $f\in\mathcal{V}(\alpha,\beta)$ does not exists.
\end{remark}
\begin{theorem}
  Let $0\leq \alpha<1$ and $\beta>1$. If a function $f$ belongs to the class $\mathcal{V}(\alpha,\beta)$, then
    $||f||<\infty$.
\end{theorem}
\begin{proof}
  Let $0\leq \alpha<1$, $\beta>1$ and $f\in\mathcal{V}(\alpha,\beta)$. Then by Lemma \ref{lem V alpha beta} and by use of definition of subordination, we have
    \begin{equation}\label{p th 2.2 1}
    \left(\frac{z}{f(z)}\right)^2f'(z)=P_{\alpha,\beta}(w(z))= 1+\frac{(\beta-\alpha)i}{\pi}\log\left(\frac{1-e^{i\phi}w(z)}{1-w(z)}\right)\quad(z\in\Delta),
  \end{equation}
  where $w$ is Schwarz function and $\phi$ is defined in \eqref{phi}. Taking logarithm on both sides of  \eqref{p th 2.2 1} and differentiating, we get
  \begin{equation}\label{p th 2.2 2}
    \frac{f''(z)}{f'(z)}=2\left(\frac{f'(z)}{f(z)}-\frac{1}{z}\right)+\frac{(\beta-\alpha)i}{\pi}
    \left[\frac{(1-e^{i\phi})w'(z)}{(1-w(z))(1-e^{i\phi}w(z))\left(1+\frac{(\beta-\alpha)i}
    {\pi}\log\left(\frac{1-e^{i\phi}w(z)}
    {1-w(z)}\right)\right)}\right].
  \end{equation}
  With a simple calculation, \eqref{p th 2.2 1} implies that
  \begin{equation}\label{identity}
    \left(\frac{f'(z)}{f(z)}-\frac{1}{z}\right)=\frac{f(z)}{z}\left(\frac{ P_{\alpha,\beta}(w(z))}{z}-1\right).
  \end{equation}
  Combining \eqref{p th 2.2 2} and \eqref{identity}, give us
  \begin{equation*}
        \frac{f''(z)}{f'(z)}=2\left(\frac{f(z)}{z}\left(\frac{ P_{\alpha,\beta}(w(z))}{z}-1\right)\right)+\frac{(\beta-\alpha)i}{\pi}
    \left[\frac{(1-e^{i\phi})w'(z)}{(1-w(z))(1-e^{i\phi}w(z))\left(1+\frac{(\beta-\alpha)i}
    {\pi}\log\left(\frac{1-e^{i\phi}w(z)}
    {1-w(z)}\right)\right)}\right]
  \end{equation*}
  It was proved in (\cite[Theorem 2.2]{KES(Siberian)}) that if $f\in\mathcal{V}(\alpha,\beta)$ where $0<\alpha\leq 1/2$ and $\beta>1$, then
  \begin{equation*}
    1-\frac{1}{\alpha}<{\rm Re}\left\{\frac{f(z)}{z}\right\}<\infty\quad(z\in\Delta).
  \end{equation*}
  Since ${\rm Re}\{z\}\leq|z|$, the last two--sided inequality means that $|f(z)/z|<\infty$ when $f\in\mathcal{V}(\alpha,\beta)$. Thus from the above we deduce that
  \begin{equation*}
    \left|\frac{f''(z)}{f'(z)}\right|<\infty\quad(z\in\Delta)
  \end{equation*}
  and concluding the proof.
\end{proof}

\end{document}